\documentclass[a4paper]{article}
\usepackage{amsmath,amsthm,amssymb,stmaryrd}
\usepackage{hyperref}
\usepackage{todonotes}

\newtheorem{theorem}{Theorem}[section]
\newtheorem{proposition}[theorem]{Proposition}
\newtheorem{lemma}[theorem]{Lemma}

\newtheorem{remark}[theorem]{Remark}
\theoremstyle{definition}

\newtheorem{definition}[theorem]{Definition}

\newcommand{\names}{\mathbb{A}}
\newcommand{\nat}{\mathbb{N}}
\newcommand{\perm}{\operatorname{Perm}}
\newcommand{\powset}{\mathcal{P}}
\newcommand{\perma}{\operatorname{Perm}(\names)}
\newcommand{\zf}{\mathbf{ZF}}
\newcommand{\izf}{\mathbf{IZF}}
\newcommand{\czf}{\mathbf{CZF}}

\newcommand{\set}{\mathbf{Set}}
\newcommand{\wlpo}{\mathbf{WLPO}}

\title{Some Brouwerian Counterexamples Regarding Nominal Sets in
  Constructive Set Theory} 

\author{Andrew W Swan}

\begin{document}

\maketitle

\begin{abstract}
  The existence of least finite support is used throughout the subject
  of nominal sets. In this paper we give some Brouwerian
  counterexamples showing that constructively, least finite support
  does not always exist and in fact can be quite badly behaved. On
  this basis we reinforce the point that when working constructively
  with nominal sets the use of least finite support should be
  avoided. Moreover our examples suggest that this problem can't be
  fixed by requiring nominal sets to have least finite support by
  definition or by using the notion of subfinite instead of finite.
\end{abstract}

\section{Constructive Set Theory}
\nocite{pittsnomsets}
\nocite{aczelrathjen}

We work in a constructive set theory such as $\czf$, as described by
Aczel and Rathjen in \cite{aczelrathjen}. A detailed knowledge of the
axioms of $\czf$ is not required for this paper; we only require the
following facts. Firstly, the underlying logic of $\czf$ is
intuitionistic, so we do not assume excluded middle as an axiom, and
so can only use instances of excluded middle that can be derived from
the other axioms. Secondly, $\czf$ includes the axiom schema of
bounded separation, so for any set $X$ and any bounded formula (i.e. a
formula where each quantifier can be written in the form
$\forall y \in z$ or $\exists y \in z$) we can form the set
$\{ x \in X \;|\; \phi \}$. Note that $\phi$ may or may not include
$x$ as a free variable. Finally, $\czf$ includes the axiom of
extensionality, so to show two sets are equal amounts to showing they
contain exactly the same elements. The results in this paper also
apply to variants of $\czf$, such as $\izf$, which is also defined in
\cite{aczelrathjen}.

We will use the following terminology conventions for concepts in
constructive set theory.

\begin{definition}
  \label{def:1}
  \leavevmode
  \begin{enumerate}
  \item We say a set $X$ has \emph{decidable equality} to mean $X$ is
    a set such that for all $x, y \in X$ we have either $x = y$ or $x
    \neq y$.
  \item If $Y \subseteq X$, we say $Y$ is a \emph{decidable subset} of
    $X$ to mean that for all $x \in X$, either $x \in Y$ or $x \notin
    Y$.
  \item The natural number $n$ is equal to the set $\{m \;|\; m < n
    \}$ and so is itself a set with $n$ elements. In particular $0$ is
    the empty set. We write $\nat$ for the set of natural numbers.
  \item A set $X$ is \emph{finite} if for some $n \in \nat$ there
    exists a bijection from $n$ to $X$.
  \item A set $X$ is \emph{finitely enumerable} if for some $n \in \nat$
    there exists a surjection from $n$ to $X$.
  \item A set $X$ is \emph{subfinite} if it is a subset of some
    finitely enumerable set.
  \item A set $X$ is \emph{infinite} if for every finite set $Y
    \subseteq X$ there exists $x \in X$ such that $x \notin Y$.
  \item We say a set $X$ is \emph{inhabited} if there exists some set
    $x$ such that $x \in X$.
  \end{enumerate}
\end{definition}

We recall the following basic results about finite sets.

\begin{proposition}
  Let $X$ be a set with decidable equality.
  \begin{enumerate}
  \item Let $Y$ be a subset of $X$. Then $Y$ is finite if and only if
    it is finitely enumerable.
  \item If a subset $Y$ of $X$ is finite then it is a decidable subset
    of $X$.
  \item Finite subsets of $X$ are closed under binary union and
    intersection.
  \end{enumerate}
\end{proposition}

\begin{proof}
  These statements can be proved by induction on the size of the
  sets. See \cite[Chapter 8]{aczelrathjenbookdraft} for more details.
\end{proof}

\begin{definition}
  \label{def:2}
  The \emph{weak limited principle of omniscience}, $\wlpo$ is the
  following axiom. Let $\alpha \colon \nat \rightarrow 2$ be any
  binary sequence. The statement $(\forall n \in \nat)\,\alpha(n) = 0$
  is either true or false.
\end{definition}

\begin{theorem}
  \label{thm:1}
  $\wlpo$ is not provable in $\czf$.
\end{theorem}

\begin{proof}
  $\wlpo$ fails in Kleene realizability models, for example, since it
  implies the existence of non computable functions. For $\czf$
  specifically, see the realizability model by Rathjen in
  \cite{rathjen06}.
\end{proof}

If we wish to show that a certain statement is not provable
constructively, one way to do this is to use it to derive a principle
(in this case $\wlpo$) that is already known to be non constructive. This
kind of proof is known as Brouwerian counterexample.

\section{Nominal Sets}

Nominal sets were introduced by Gabbay and Pitts to give an abstract
notion of names and binding. Since then they have been applied many
places in computer science (see \cite{pittsnomsets}). Recently they
have seen applications in the semantics of homotopy type theory (see
\cite{pittsnompcs} and \cite{swannomawfs}).

We recall the following definitions from \cite{pittsnomsets}. Let
$\names$ be an infinite set with decidable equality (which we will
refer to as the set of \emph{names}). Write $\perma$ for the group of
finite permutations (that is, permutations $\pi$ such that $\pi(a) =
a$ for all but finitely many $a \in \names$). Note that since we are
assuming $\names$ has decidable equality, $\perma$ is precisely the
group generated by transpositions (i.e. swapping two elements and fixing
everything else). Recall that a $\perma$-set is a set $X$, together
with an action of $\perma$ on $X$, or equivalently a presheaf over
$\perma$ when viewed as a one object category in the usual way.

\begin{definition}[Pitts, Gabbay]
  \leavevmode
  \begin{enumerate}
  \item Let $X$ be a $\perma$-set (writing $\cdot$ for
    the action) and let $x \in X$. We say $A \subseteq \names$ is a
    \emph{support} for $x$ if whenever $\pi(a) = a$ for all $a \in A$,
    also $\pi \cdot x = x$.
  \item Let $X$ be a $\perma$-set and $x \in X$. We say $x$ is
    \emph{equivariant} if $\pi \cdot x = x$ for all $\pi \in \perma$,
    or equivalently if $\emptyset$ is a support for $x$.
  \item Let $X$ and $Y$ be $\perma$-sets. A function $f: X \rightarrow
    Y$ is \emph{equivariant} if it is a morphism in the category of
    $\perma$ sets, or equivalently if it is equivariant as an element
    of the exponential $Y^X$ in the category of $\perma$ sets, which
    is described explicitly as the set of functions $X$ to $Y$ with
    action given by conjugation.
  \item A \emph{nominal set} is a $\perma$-set $X$,
    such that for every $x \in X$, there exists a finite set $A
    \subseteq \names$ such that $A$ is a support for $x$.
  \end{enumerate}
\end{definition}

\begin{proposition}
  \label{prop:finsuppbinint}
  Let $X$ be a $\perma$-set and let $A$ and $B$ be finite supports of
  $x \in X$. Then $A \cap B$ is also a finite support of $x$. Hence
  also, whenever $A_1,\ldots, A_n$ are finite supports of $x \in X$ so
  is $A_1 \cap \ldots \cap A_n$.
\end{proposition}

\begin{proof}
  The proof of \cite[Proposition 2.3]{pittsnomsets} is already
  constructive as stated, but we include a proof here anyway to
  illustrate the proof techniques used for this kind of result.
  
  To show $A \cap B$ is a support for $x$, we only have to show that
  for each  finite permutation $\pi$ that fixes each element of $A \cap
  B$, we have $\pi \cdot x = x$.

  First note that by \cite[Theorem 1.15]{pittsnomsets} we can find
  $a_1,\ldots, a_n$ and $a_1', \ldots, a_n'$ such that $\pi(a_i) \neq
  a_i \neq a_i' \neq \pi(a_i')$ and $\pi$ decomposes as a product of
  transpositions
  \begin{equation}
    \label{eq:7}
    \pi = (a_1 \; a_1') \circ \ldots \circ (a_n \; a_n')
  \end{equation}
  (The argument in \cite{pittsnomsets} is an explicit argument by
  induction which is constructive as stated, using the decidability of
  equality for $\names$ and in particular decidability of finite
  subsets.)

  Since $\pi$ fixes each element of $A \cap B$ and $\pi(a_i) \neq a_i$
  we have $a_i \notin A \cap B$ and similarly $a_i' \notin A \cap B$.
  Now note if each transposition $(a_i \; a_i')$ fixes $x$, then so
  does $\pi$.  Hence the problem is reduced to showing that if $a, b
  \notin A \cap B$ then $(a \; b) \cdot x = x$. Since $A$ and $B$ are
  finite, they are decidable subsets of $\names$. Hence we can split
  into the 16 cases depending on whether or not $a$ and $b$ are
  elements of $A$ and $B$. By assumption $a, b \notin A \cap B$ so we
  can eliminate these 7 cases immediately. If $a, b \notin A$ then $(a
  \; b)$ fixes $A$ which is a support for $x$ so already $(a \; b)
  \cdot x = x$. Similarly for $a, b \notin B$. So we can eliminate
  another 7 cases. This only leaves the case $a \in A$ and $b \in B$
  and the case $b \in A$ and $a \in B$. Without loss of generality, we
  only have to check the case $a \in A$ and $b \in B$. Since $\names$
  is infinite there exists $c \notin A \cup B$. Then we decompose $(a
  \; b)$ as follows.
  \begin{equation}
    \label{eq:8}
    (a \; b) = (b \; c) \circ (a \; c) \circ (b \; c)
  \end{equation}
  Then $(b \; c)$ fixes $A$ and $(a \; c)$ fixes $B$. Hence each
  transposition in the composition fixes $x$ and so $(a \; b)$ does
  also.
\end{proof}

\begin{proposition}
  \label{prop:arbintsupp}
  In $\zf$ the following holds. Suppose that $X$ is a $\perma$-set and
  $(A_i)_{i \in I}$ is an indexed family of finite supports for $x \in
  X$ with $I$ inhabited. Then $\bigcap_{i \in I} A_i$ is also a finite
  support for $x$.
\end{proposition}

\begin{proof}
  Since $I$ is inhabited, there exists some $i_0 \in I$. Note that we
  have
  \begin{align}
    \bigcap_{i \in I} A_i &= \bigcap_{i \in I} (A_{i_0} \cap A_i) \\
    &= \bigcap \{ B \in \powset(A_{i_0}) \;|\; (\exists i \in I)\,B =
    (A_i \cap A_{i_0}) \}
  \end{align}

  In classical logic the class of finite sets is closed under power
  sets and subsets, so the set $\{ B \in \powset(A_{i_0}) \;|\;
  (\exists i \in I)\,B = A_i \}$ is a finite collection of finite
  sets. The result now follows from proposition
  \ref{prop:finsuppbinint}.
\end{proof}

\begin{proposition}
  In $\zf$, for any nominal set $X$, and any $x \in X$ there exists a
  least finite support of $x$, which is defined as below.
  \begin{equation}
    \label{eq:6}
    \operatorname{Supp}(x) := \bigcap \{ A \subseteq \names \;|\; A \text{ is
      a finite support for } x\}
  \end{equation}
\end{proposition}

\begin{proof}
  Since $X$ is a nominal set, $x$ must have at least one finite
  support. Now applying proposition \ref{prop:arbintsupp} the
  definition given is a finite support of $x$. However, it is a subset
  of any finite support of $x$ by definition.
\end{proof}

The concept of least finite support is used throughout standard
presentations of nominal sets such as \cite{pittsnomsets}. However, we
will see that the theorems of $\zf$ above fail quite badly in a
constructive setting, so least finite support should not be used
constructively.

Much of the basic theory of nominal sets has been proved
constructively (and avoiding least finite support) by Choudhury in
\cite{choudhuryconnom}, with proofs verified electronically in the
Agda proof assistant.

\section{The Counterexamples}
We aim towards a theorem providing our first example of a nominal set
where least finite support does not provably exist. We use a concrete
definition that only requires the natural number object $\mathbb{N}$,
the terminal object, the nominal set of names $\names$, binary
coproducts and a single instance of exponentiation.

\begin{definition}
  Let $X$ be a nominal set. We say a \emph{support function} is a
  function $S \colon X \rightarrow \powset_\mathrm{fin} \names$, such
  that for every $x \in X$, $S(x)$ is a finite support of $x$.
\end{definition}

\begin{proposition}
  Let $X$ be a nominal set where for every $x \in X$, a least finite
  support of $x$ exists. Then $X$ admits a support function.
\end{proposition}

\begin{proof}
  Least finite support is unique if it exists, so it gives us a well
  defined function.
\end{proof}

Recall that exponentials in nominal sets are implemented as
follows. We first define the exponential $Y^X$ in $\perma$-sets to be
the exponential in $\set$ (i.e. the set of functions from $X$ to $Y$)
together with action given by conjugation (i.e. $\pi \cdot f (x) :=
\pi \cdot (f(\pi^{-1} \cdot x))$). We then take the nominal set
exponential to be the subobject of $Y^X$ consisting of those functions
$f$ for which a finite support exists. One can show that this is still
an exponential constructively either directly, or by adapting the
proof of \cite[Theorem 2.19]{pittsnomsets}.

We also recall the following basic facts about nominal sets.
\begin{enumerate}
\item The natural number object in nominal sets is just $\nat$ with
  the trivial action ($\pi \cdot n = n$ for all $n$).
\item The terminal object $1$ in nominal sets is any singleton set
  with the trivial action.
\item The set of names, $\names$, can be viewed as a nominal set in a
  canonical way by taking $\pi \cdot a$ to be $\pi(a)$.
\item Coproducts are implemented as disjoint union, as in $\set$, with
  action defined componentwise.
\end{enumerate}

\begin{theorem}
  \label{thm:nocc}
  Suppose that the following nominal set admits a support function.
  \begin{equation}
    (\names + 1)^\nat
  \end{equation}
  Then $\wlpo$ follows. (And hence one cannot show constructively that
  this nominal set admits a support function, or that every element
  has least finite support).
\end{theorem}

\begin{proof}
  Assume that $S$ is a support function for $(\names + 1)^\nat$.

  We write $\ast$ for the unique element of $1$.

  Let $\underline{\ast} \colon \nat \rightarrow \names + 1$
  be the function constantly equal to $\ast$. Let
  $a \in \names \setminus S(\underline{\ast})$.

  Now for any $\alpha \colon \nat \rightarrow 2$, we consider
  $\alpha_a \colon \nat \rightarrow \names + 1$ as below.
  \begin{equation}
    \label{eq:10}
    \alpha_a(n) :=
    \begin{cases}
      \ast & \alpha(n) = 0 \\
      a & \alpha(n) = 1
    \end{cases}
  \end{equation}

  First note that $\alpha_a$ is equivariant if and only if we have
  $(\forall n \in \nat)\,\alpha(n)=0$. Certainly if $\alpha(n) = 0$
  for every $n$ then $\alpha_a$ is equivariant. For the converse, let
  $a' \in \names \setminus \{a\}$ and note that
  $(a \; a')\cdot \alpha_a = \alpha_a$, by equivariance. Then for
  every $n \in \nat$, we have
  $(a \; a') \cdot \alpha_a(n) = \alpha_a(n)$. Now for every $n$, we
  have $\alpha(n) = 0$ or $\alpha(n) = 1$, because $\nat$ has
  decidable equality, but if $\alpha(n) = 1$, then we would get
  \begin{align}
    \alpha_a(n) &= a \\
    \alpha_a(n) &= (a \; a') \cdot \alpha_a(n) \\
                &= (a \; a') \cdot a \\
                &= a' \neq a
  \end{align}
  Hence we derive $\alpha(n) = 0$ for all $n$, as required.

  Since $S(\alpha_a)$ is finite we have that either
  $a \in S(\alpha_a)$ or $a \notin S(\alpha_a)$. We split into the two
  cases. If $a \in S(\alpha_a)$, then
  $S(\alpha_a) \neq S(\underline{\ast})$. Since $S$ is a function, this
  gives $\alpha_a \neq \underline{\ast}$. But then we must have
  $\neg (\forall n \in \nat)\,\alpha(n) = 0$. On the other hand, if
  $a \notin S(\alpha_a)$, then $\emptyset$ is a support for $\alpha_a$
  (since it is the intersection of $S(\alpha_a)$ and $\{a\}$). Hence
  we have $(\forall n \in \nat)\,\alpha(n) = 0$. So applying this for
  any $\alpha$, we get $\wlpo$.  
\end{proof}

In the next lemma we give another example which is less concrete, but
in some ways more useful. A key point is that in the previous example
there was an element that was equivariant if and only if the statement
$(\forall n \in \nat)\,\alpha(n) = 0$ holds, whereas in the next
lemma we will see that something similar can be done more generally
for any bounded formula.

\begin{lemma}
  \label{lem:nointer}
  Let $\phi$ be a formula with only bounded quantifiers. (The reason
  for this restriction is that $\czf$ has separation only for bounded
  formulas; over $\izf$ $\phi$ can by any formula). There is a nominal
  set $X$ and an element $\bar{a} \in X$ such that $\bar{a}$ is
  equivariant iff $\phi$ holds.

  Also there is a set $\mathcal{S}$ such that
  \begin{enumerate}
  \item $\mathcal{S}$ is inhabited
  \item Each $S \in \mathcal{S}$ is a finite support of $\bar{a}$
  \item If $\bigcap \mathcal{S}$ is a support for $\bar{a}$, then
    $\neg \neg \phi \rightarrow \phi$.
  \end{enumerate}

  Furthermore, if $\phi$ is of the form $\psi \vee \neg \psi$ where
  $\neg \neg \psi \rightarrow \psi$, then there are subfinite supports
  $S_1$ and $S_2$ of $\bar{a}$ such that $S_1 \cap S_2 = \emptyset$.
  (Note also that $\neg \neg (\psi \, \vee \, \neg \psi)$ holds for
  any formula $\psi$ in intuitionistic logic, so if
  $\bigcap \mathcal{S}$ is a support for $\bar{a}$, we derive in this
  case $\psi \vee \neg \psi$.)
\end{lemma}

\begin{proof}
  Fix a bounded formula, $\phi$. For any $a \in \names$, define
  \begin{equation}
    \label{eq:2}
    \bar{a} := \{ x \in \{a\} \;|\; \phi \}
    \cup (\names \setminus \{a\})
  \end{equation}
  We then define $X$ to be the set
  \begin{equation}
    \label{eq:3}
    X := \{ \bar{a} \;|\; a \in \names \}
  \end{equation}
  We define the action of $\pi \in \perm(\names)$ by
  \begin{equation}
    \label{eq:4}
    \pi \cdot \bar{a} := \{ \pi(x) \;|\; x \in \bar{a} \} =
    \overline{\pi(a)}
  \end{equation}

  Suppose that $\phi$ holds. Then $\bar{a} = \names$ and hence for all
  $\pi \in \perma$ we have $\pi \cdot \bar{a} = \bar{a}$. Therefore
  $\bar{a}$ is equivariant.

  Conversely, suppose that $\bar{a}$ is equivariant. Then for any
  $a' \neq a$ we have $(a \; a') \cdot \bar{a} = \bar{a}$. Certainly
  $a' \in \bar{a}$, so we deduce $a \in \bar{a}$, and so $\phi$ must
  hold.

  We have now shown that $\bar{a}$ is equivariant iff $\phi$ holds.
  
  Now define $\mathcal{S}$ as follows:
  \begin{equation}
    \label{eq:5}
    \mathcal{S} := \{ x \in \{\emptyset \} \;|\; \phi \}
    \cup \{\{a\}\}
  \end{equation}

  $\mathcal{S}$ is inhabited since $\{a\} \in \mathcal{S}$.

  For any $S \in \mathcal{S}$ we have that either $S = \emptyset$ and
  $\phi$ holds, or $S = \{a\}$. In both cases $S$ is a finite support
  of $\bar{a}$.

  Now assume that $\bigcap \mathcal{S}$ is a support for
  $\bar{a}$. Since $\phi \; \rightarrow \emptyset \in \mathcal{S}$ we
  have $\phi \; \rightarrow \; \bigcap \mathcal{S} = \emptyset$, and
  so
  $\neg \neg \phi \; \rightarrow \; \neg \neg \, \bigcap \mathcal{S} =
  \emptyset$.
  However, from $\neg \neg \, \bigcap \mathcal{S} = \emptyset$ we can
  derive $\bigcap \mathcal{S} = \emptyset$: in general the statement
  that a set $X$ is empty is stable under double negation which is
  easy to show noting that $X$ is empty precisely when
  $(\forall x \in X)\,\bot$. But then $\bigcap \mathcal{S}$ being a
  support for $\bar{a}$ says precisely that $\emptyset$ is a support
  for $\bar{a}$, which we have already shown implies $\phi$. Putting
  this together gives us that if $\bigcap \mathcal{S}$ is a support
  for $\bar{a}$ then $\neg \neg \phi \rightarrow \phi$, as required.

  Now suppose that $\phi$ is of the form $\psi \vee \neg \psi$ where
  $\neg \neg \psi \rightarrow \psi$ and define $S_1$ and $S_2$ as
  follows:
  \begin{align}
    S_1 &:= \{ a \in \{a\} \;|\; \psi \} \\
    S_2 &:= \{ a \in \{a\} \;|\; \neg \psi \}
  \end{align}
  Suppose that $\pi$ is a finite permutation such that $\pi(x) = x$
  whenever $x \in S_1$. Note that this precisely says that if $\psi$
  holds then $\pi(a) = a$. Since equality in $\names$ is decidable, we
  have $\pi(a) = a$ or $\pi(a) \neq a$. In the former case, we easily
  have that $\pi \cdot \bar{a} = \bar{a}$. In the latter case, we
  deduce $a \notin S_1$ and so $\neg \psi$ and thereby $\phi$ which is
  equal to $\psi \vee \neg \psi$. Since this implies $\bar{a}$ is
  equivariant, we also have in this case that
  $\pi \cdot \bar{a} = \bar{a}$. Hence $S_1$ is a support for
  $\bar{a}$. For $S_2$ we start the same as before. However, if $\pi$
  fixes the elements of $S_2$ and $\pi(a) \neq a$ we only derive
  $\neg \neg \psi$. At this point we use the assumption that
  $\neg \neg \psi \rightarrow \psi$ and then continue the same as
  before.

  Finally, we easily have $S_1 \cap S_2 = \emptyset$ since $\psi
  \wedge \neg \psi$ is false.
\end{proof}

We now apply the above examples to get a number of independence
results.

\begin{theorem}
  \label{thm:main}
  Suppose that any one of the statements below holds in general for
  all nominal sets $X$ and $Y$ and all elements $x$ of $X$. Then $\wlpo$
  follows. Hence none of these statements are constructively provable
  in general.
  \begin{enumerate}
  \item If $X$ and $Y$ admit support functions then so does
    $Y^X$. \label{suppfexp}
  \item If every element of $X$ and every element of $Y$ has a least
    finite support, then so does every element of $Y^X$. \label{lsuppexp}
  \item The intersection of all finite supports of $x$ is
    finite. \label{intfin}
  \item The intersection of all finite supports of $x$ is a support of
    $x$. \label{intsupp}
  \item $x$ has a least finite support. \label{leastfinsupp}
  \item $x$ has a least subfinite support. \label{leastsubsupp}
  \item The binary intersection of two subfinite supports of $x$ is a
    support of $x$. \label{bintsubfin}
  \item For every subfinite support $S$ of $x$, there is a finite
    support $S'$ of $x$ such that $S' \subseteq
    S$. \label{finsubsubfin}
  \end{enumerate}
\end{theorem}

\begin{proof}
  For \ref{suppfexp} and \ref{lsuppexp}, we apply theorem
  \ref{thm:nocc} directly, noting that every element of $\names + 1$
  and every element of $\nat$ has least finite support.
  
  For \ref{intfin}, we take $X := (\names + 1)^\nat$ and again apply
  theorem \ref{thm:nocc}. Let $\alpha_a$ be as in the proof of theorem
  \ref{thm:nocc}.

  We first check that the intersection of all finite supports of
  $\alpha_a$ is equal to the following set.
  \begin{equation}
    \label{eq:11}
    L := \{x \in \{a\} \;|\; \neg\,(\forall n \in
    \nat)\,\alpha(n) = 0 \}
  \end{equation}
  To show the two sets are equal, we need to show that they have the
  same elements. Suppose first that $x \in L$. First note that we must
  have $x = a$ and that $\neg \, (\forall n \in \nat)\,\alpha(n) = 0$
  is true. We want to show that $a$ lies in every finite support of
  $\alpha_a$. Let $S$ be a finite support of $\alpha_a$. We want to
  show $a \in S$, but since $S$ is a finite, and so decidable subset
  of $\names$, it suffices to show $\neg \neg a \in S$. Now if we had
  $a \notin S$, then we would have $S \cap \{a\} = \emptyset$. But
  then $\emptyset$, as the intersection of two finite supports would
  itself be a finite support, making $\alpha_a$ equivariant and so we
  would get $(\forall n \in \nat)\, \alpha(n) = 0$, contradicting
  $\neg (\forall n \in \nat)\,\alpha(n) = 0$. We deduce
  $\neg \neg a \in S$ and so $a \in S$, as required.

  Now conversely, assume that $x$ belongs to every finite support of
  $\alpha_a$. Since $\{a\}$ is a finite support we must have $x =
  a$. Furthermore, since the intersection of all finite supports is
  inhabited, $\emptyset$ cannot be a finite support. Hence $\alpha_a$
  is not equivariant, and so we derive $\neg (\forall n \in
  \nat)\,\alpha_a(n) = 0$. But then we have $x \in L$.

  We have now verified that $L$ is equal to the intersection of all
  finite supports.
  
  If $L$ was finite, it would also be decidable as a subset of
  $\names$. But then we could decide whether or not $a \in L$ and then
  recalling that the statement $(\forall n \in \nat)\,\alpha(n)=0$ is
  stable under double negation, derive
  $(\forall n \in \nat)\,\alpha(n)=0 \;\vee\; \neg (\forall n \in
  \nat)\,\alpha(n)=0$.
  If \ref{intfin} held in general we could show this for any $\alpha$
  and so derive $\wlpo$.
  
  For \ref{intsupp} to \ref{finsubsubfin} we use lemma
  \ref{lem:nointer} as follows.

  For \ref{intsupp}, let $\phi$ be any bounded formula. Note that the
  intersection of all least finite supports of $\bar{a}$ has to be a
  subset of $\bigcap \mathcal{S}$ and so if it is a support then so is
  $\bigcap \mathcal{S}$. Hence if the intersection of all finite
  supports is a support for all elements of all nominal sets we derive
  $\neg \neg \phi \rightarrow \phi$ for any bounded formula $\phi$
  (and therefore also $\phi \vee \neg \phi$ for any bounded
  $\phi$). In particular this gives $\wlpo$, but it is of course much
  stronger.

  For \ref{leastfinsupp}, we can use either theorem \ref{thm:nocc} or
  lemma \ref{lem:nointer}. There is also another example in
  \cite[Section 6.1]{swannomawfs} based on the nerve of a metric
  space. Note that we again can in fact derive
  $\neg \neg \phi \rightarrow \phi$ for all bounded formulas when we
  use lemma \ref{lem:nointer}.

  For \ref{leastsubsupp}, note that any least subfinite support has to
  be in particular a subset of each finite support, and so also a
  subset of the intersection of all finite supports. But this would
  imply that the intersection of all finite supports is a
  support. Hence this part follows from part \ref{intsupp}. It will
  also follow from \ref{bintsubfin}, which we show next.
  
  For \ref{bintsubfin}, take $\psi$ to be any bounded formula such
  that $\neg \neg \psi \rightarrow \psi$ and take $\phi$ to be
  $\psi \vee \neg \psi$, then apply the last part of lemma
  \ref{lem:nointer}. If $S_1$ and $S_2$ are as in the statement of
  lemma \ref{lem:nointer} and $S_1 \cap S_2$ is a support of
  $\overline{a}$, then $\psi \vee \neg \psi$ holds. In particular,
  taking $\psi$ to be $(\forall n \in \nat)\,\alpha(n) = 0$ for
  arbitrary binary sequences $\alpha$ gives us $\wlpo$.

  For \ref{finsubsubfin}, note that if there are finite supports
  $S_1'$ and $S_2'$ such that $S_1' \subseteq S_1$ and
  $S_2' \subseteq S_2$ then $S_1' \cap S_2'$ would also be a finite
  support. However, $S_1' \cap S_2' \subseteq S_1 \cap S_2$, so this
  would imply that $S_1 \cap S_2$ is a support and so we can again
  apply lemma \ref{lem:nointer}. (Note also that $L$ in the proof of
  part \ref{intfin} is a subfinite support, and so that gives another
  proof.)
\end{proof}

\begin{remark}
  For \ref{intsupp}, \ref{leastfinsupp} and \ref{leastsubsupp} of
  theorem \ref{thm:main}, we could show not just $\wlpo$, but
  $\phi \vee \neg \phi$ for all bounded formulas $\phi$, which is much
  stronger. This schema is sometimes referred to as \emph{restricted
    excluded middle}, $\mathbf{REM}$.
\end{remark}

\begin{remark}
  Formally speaking, a subfinite subset $A$ of $\names$ is just a
  subset of $\names$ which is subfinite. We know that $A \subseteq F$
  for some finite set $F$. We don't necessarily have that $F$ is
  itself a subset of $\names$. However, if we take ``subfinite subset
  of $\names$'' to mean a set satisfying the stronger condition, that
  it is a subset of a finite set that is itself a subset of $\names$,
  then lemma \ref{lem:nointer} and so also theorem \ref{thm:main}
  still hold.
\end{remark}

\section{Acknowledgement}
This is based on work carried out at the University of Leeds under
EPSRC grant EP/K023128/1.

\bibliographystyle{abbrv}
\bibliography{mybib}{}

\begin{thebibliography}{1}

\bibitem{aczelrathjen}
P.~Aczel and M.~Rathjen.
\newblock Notes on constructive set theory.
\newblock Technical Report~40, Institut Mittag-Leffler, 2001.

\bibitem{aczelrathjenbookdraft}
P.~Aczel and M.~Rathjen.
\newblock Notes on constructive set theory.
\newblock Book draft available at
  \url{http://www1.maths.leeds.ac.uk/~rathjen/book.pdf}, 2010.

\bibitem{choudhuryconnom}
P.~Choudhury.
\newblock Constructive representation of nominal sets in agda.
\newblock Master's thesis, Robinson College, University of Cambridge, June
  2015.
\newblock Available at
  \url{https://www.cl.cam.ac.uk/~amp12/agda/choudhury/choudhury-dissertation.pdf}.

\bibitem{pittsnomsets}
A.~M. Pitts.
\newblock {\em Nominal Sets: Names and Symmetry in Computer Science}, volume~57
  of {\em Cambridge Tracts in Theoretical Computer Science}.
\newblock Cambridge University Press, 2013.

\bibitem{pittsnompcs}
A.~M. Pitts.
\newblock Nominal presentation of cubical sets models of type theory.
\newblock In H.~Herbelin, P.~Letouzey, and M.~Sozeau, editors, {\em 20th
  International Conference on Types for Proofs and Programs (TYPES 2014)},
  Leibniz International Proceedings in Informatics (LIPIcs), Dagstuhl, Germany,
  2015. Schloss Dagstuhl--Leibniz-Zentrum fuer Informatik.

\bibitem{rathjen06}
M.~Rathjen.
\newblock Realizability for constructive {Z}ermelo-{F}raenkel set theory.
\newblock In V.~Stoltenberg-Hansen and J.~V\"{a}\"{a}n\"{a}nen, editors, {\em
  Logic Colloquium '03}. Association for Symbolic Logic, 2006.

\bibitem{swannomawfs}
A.~W. Swan.
\newblock An algebraic weak factorisation system on 01-substitution sets:\ a
  constructive proof.
\newblock arXiv:1409.1829, September 2014.

\end{thebibliography}

\end{document}